\documentclass[final,leqno,onefignum,onetabnum]{siamltex1213}

 \usepackage{amsmath}
 \usepackage{amssymb}

\newtheorem{remark}{Remark}

\title{Weighted Orthogonal Polynomials-Based Generalization of Wirtinger-Type Integral Inequalities for Delayed Continuous-Time Systems\thanks{This work was supported in part by the National Na\-t\-u\-r\-al Science Foundation of China (11371006 and 61203005), the Natural Science Foundation of Heilong\-jiang Province
(QC2013C068, F201326 and A201416),
the Fund of Heilongjiang Education Committee (12541\-603),
and the Postdoctoral Science-research
Developmental Foundation of Heilongjiang Province (LBH-Q12130).}}

\author{Xian Zhang, Yuanyuan Han, Yantao Wang, Cheng Gong (Corresponding author)\thanks{School of Mathematical Science, Heilongjiang University, Harbin, 150080, P. R. China.
(\email{xianzhang@ieee.org, 1056954098@qq.com, mengxiang4692@sina.com, gongcheng2004@126.com}).}}

\begin{document}
\maketitle
\slugger{simax}{xxxx}{xx}{x}{x--x}

\begin{abstract}
 In the past three years, many researchers have proven and/or employed some Wirtinger-type integral inequalities
 to establish less conservative stability criteria for delayed continu\-ous-time systems.
 In this present paper, we will investigate weighted orthogonal polynomials-based integral inequalities
 which is a generalization of the existing Jensen's inequalities and Wirtinger-type integral inequalities.
\end{abstract}

\begin{keywords}
Wirtinger-type integral inequalities (WTIIs); delayed continuous-time systems; weighted orthogonal polynomials (WOPs).
\end{keywords}

\begin{AMS}
15A45, 34K38, 35A23
\end{AMS}

\pagestyle{myheadings}
\thispagestyle{plain}
\markboth{X., Y. Y. Han, Y. T. Wang, C. Gong}{WOPs-Based Generalization of Wirtinger-Type Integral Inequalities}

\section{Introduction}\label{section-1}

Time delays are inherent in many nature's processes and systems,
for example, spread of infectious diseases and epidemics \cite{wang2014global},
population dynamics systems \cite{JMAA-2015-415},
neural networks \cite{CNSNS-2013-1246,JFI-2013-966},
vehicle active suspension \cite{JSSMS-C-2014-1206}, and
biological and chemical systems \cite{N-2014-105,Zhang.Wu.Zou-IEEE-T-CBB}.
Since time delays are generally regarded as one of main sources of instability and poor performance \cite{N-2015-199,Int.j.RNC-2011-preprint}, the stability analysis issue of time-delay systems is important and has received considerable attention
(see \cite{IEEE-T-CBB-2015-398,A-2015-189,JFI-2015-1378,NCA-2013,NA-RWA-2012-2188} and the references therein).

Most of the results on stability analysis of delayed con\-ti\-nu\-ous-time
systems are obtained by the
Lyapunov-Krasovskii functional (LKF) approach \cite{Gu.Kharitonov.Chen(2003)}.
A key step of the LKF approach is how to construct LKF and to bound its derivative.
It is well-known that an indispensable part of LKF is some integer items like
\begin{eqnarray}\label{15aug261}
\mathcal{I}_m(w_t):=\int_{a}^b(s-a)^mw_t^\mathrm{T}(s)Rw_t(s)\mathrm{d}s,\ t\ge 0,\nonumber
\end{eqnarray}
where $w_t:[a,b]\rightarrow \mathbb{R}^n$ is defined by $w_t(s)=w(t+s)$ for all $s\in [a,b]$, $w:[0,+\infty)\rightarrow \mathbb{R}^n$ is a continuous function, $R$ is a real symmetric positive definite matrix, and $m$ is a nonnegative integer.
It is clear that $\mathcal{I}_0(w_t)=\int_a^bw_t^\mathrm{T}(s)Rw_t(s)\mathrm{d}s$ and
\begin{eqnarray}\label{15sep11222}
\mathcal{I}_m(w_t)=m!\!\!\int_{a}^b\!\int_{\theta_1}^b\!\!\cdots\!\!\int_{\theta_m}^b\!\!w_t^\mathrm{T}(s)Rw_t(s)\mathrm{d}s\mathrm{d}\theta_m\cdots\mathrm{d}\theta_{1}\nonumber
\end{eqnarray}
for $m\ge 1$.
Since
\begin{eqnarray}\label{15sep081}
\frac{\mathrm{d}}{\mathrm{d}t}\mathcal{I}_0(w_t)=w_t^\mathrm{T}(b)Rw_t(b)-w_t^\mathrm{T}(a)Rw_t(a)\nonumber
\end{eqnarray}
and
\begin{eqnarray}\label{15aug263}
\frac{\mathrm{d}}{\mathrm{d}t}\mathcal{I}_m(w_t)\hspace*{-.5mm}=\hspace*{-.5mm}(b\hspace*{-.5mm}-\hspace*{-.5mm}a)^mw_t^\mathrm{T}\hspace*{-.5mm}(b)Rw_t\hspace*{-.5mm}(b)\hspace*{-.5mm}-\hspace*{-.5mm}m\mathcal{I}_{m-1}\hspace*{-.5mm}(w_t),\ m\hspace*{-.5mm}\ge \hspace*{-.5mm}1,\nonumber
\end{eqnarray}
the conservativeness of resulting stability criterion relies mainly on the lower bound to $\mathcal{I}_{m-1}(w_t)$ for $m\ge 1$. Usually, the so-called Jensen's inequalities (JIs)  \cite{Gu(2000),IJRNC-2009-1364,AMC-2013-714} are applied to bound $\mathcal{I}_k(w_t)$ for any nonnegative integer $k$.

Recently, some new integral inequalities, spectrally Wir\-ti\-nger-type integral inequalities (WTIIs), have been proposed to improve Jensen's inequalities (i.e., to give more accurate lower bounds of $\mathcal{I}_m(w_t)$ or $\mathcal{I}_m(\dot{w}_t)$) (see \cite{Gu(2000),IJRNC-2009-1364,AMC-2013-714,A-2013-2860,ECC-2014-448,CDC-2013-946,SCL-2015-1,ECC-2013,
A-2015-204,JFI-2015stability,IEEE-T-AC-2015-free,A-2015-189,CNSNS-2013-1246,DDNS-2013-793686,JFI-2015-1378,JFI-2015enhanced,IFAC-2012,NN-2014-57,IET-CTA-2014-1054,IEEE-T-CBB-2015-398,IFAC-2014} and the references therein).
It is shown by Gyurkovics \cite{A-2015-44} that the lower bound of $\mathcal{I}_0(\dot{w}_t)$ given in \cite{A-2013-2860}
 is more accurate than ones in \cite{CNSNS-2013-1246,DDNS-2013-793686}, while the estimations to $\mathcal{I}_0(\dot{w}_t)$ obtained in \cite{A-2013-2860,IEEE-T-AC-2015-free} are equivalent.

In this present paper, we aim in reducing the conservativeness of LKF approach by investigating new integral inequalities based on weighted orthogonal polynomials (WOPs) which is a generalization of those JIs and WTIIs mentioned above as special cases.

This paper is organized as follows: In Section \ref{15sep10-section-1}, we will first introduce a class of WOPs, and thereby investigate WOPs-based inequality inequalities. Discussions of the relation between the WOPs-based inequality inequalities and the JIs and WTIIs in
\cite{Gu(2000),IJRNC-2009-1364,AMC-2013-714,A-2013-2860,ECC-2014-448,CDC-2013-946,SCL-2015-1,ECC-2013,
A-2015-204,JFI-2015stability,IEEE-T-AC-2015-free,A-2015-189,CNSNS-2013-1246,DDNS-2013-793686,JFI-2015-1378,JFI-2015enhanced,IFAC-2012}
 will be presented in Section\ref{15sep11-section-1}.
 We will conclude  the results of this paper in Section \ref{15jun06-section-1}.

{\bf\emph{Notations:}} The notations used throughout this paper are fairly standard.
Let $\mathbb{R}^{n\times m}$ be the set of all $n\times m$ matrices over the real number field $\mathbb{R}$.
For a matrix $X\in\mathbb{R}^{n\times n}$, the symbols $X^{-1}$ and $X^\mathrm{T}$ denote the inverse and transpose of $X$, respectively.
Set $\mathbb{R}^n=\mathbb{R}^{n\times 1}$ and $X^{-\mathrm{T}}=(X^{-1})^\mathrm{T}$.
The \emph{Kronecker product}, $A\otimes B$, of two matrices $A=[a_{ij}]\in\mathbb{R}^{m\times n}$ and $B\in\mathbb{R}^{p\times q}$ is the $mp\times nq$ block matrix:
\begin{eqnarray}\label{15sep111}
\begin{bmatrix}
a_{11} \mathbf{B} & \cdots & a_{1n}\mathbf{B} \\
\vdots & \ddots & \vdots \\
a_{m1} \mathbf{B} & \cdots & a_{mn} \mathbf{B}
\end{bmatrix}.\nonumber
\end{eqnarray}
Denote by $\mathrm{diag}(\cdots)$ and $\mathrm{col}(\cdots)$ the (block) diagonal matrix and column matrix formed by the elements in brackets, respectively.

\section{WOPs-based integral inequalities}\label{15sep10-section-1}

In this section we will investi\-gate novel WOPs-based integral inequalities, which is a generalization of many JIs and WTIIs in literature.

\subsection{WOPs}

If $p(s)=\sum_{k=0}^{\mathcal{N}}a_ks^k$ and $a_{\mathcal{N}}\not= 0$, then we say $p(s)$ is a polynomial of degree $\mathcal{N}$.
Let $\mathbb{R}[s]_{\mathcal{N}}$ denote the linear space of polynomials with real coefficients of degree not exceeding $\mathcal{N}$.
Set $f_k(s)=(s-a)^k$, $k=0,1,2,\dots,\mathcal{N}$.
Then $\{f_k(s)\}_{k=0}^\mathcal{N}$ is a basis of $\mathbb{R}[s]_{\mathcal{N}}$.
For an arbitrary but fixed nonnegative integer $m$,
define an inner product, $(\cdot,\cdot)_{m}$, on $\mathbb{R}[s]_n$ by
\begin{eqnarray}\label{15aug161}
(p(s),q(s))_{m}=\int_{a}^b(s-a)^mp(s)q(s)\mathrm{d}s
\end{eqnarray}
for any $p(s),q(s)\in \mathbb{R}[s]_{\mathcal{N}}$.
Let $\{p_{km}(s)\}_{k=0}^\mathcal{N}$ be the orthogonal basis of $\mathbb{R}[s]_\mathcal{N}$ which is obtained by applying the Gram-Schmidt orthogonalization process to the basis $\{f_k(s)\}_{k=0}^\mathcal{N}$, that is,
\begin{eqnarray}\label{l062}
p_{0m}(s)&\hspace*{-1mm}=&\hspace*{-1mm}f_0(s),\nonumber\\
p_{im}(s)&\hspace*{-1mm}=&\hspace*{-1mm}f_i(s)-\sum_{j=0}^{i-1}\frac{g_{ijm}}{\chi_{jm}}p_{jm}(s),\ i=1,2,\dots,\mathcal{N},
\end{eqnarray}
where
\begin{eqnarray}\label{l072}
g_{ijm}=(f_i(s),p_{jm}(s))_m,\ \chi_{jm}=(p_{jm}(s),p_{jm}(s))_m.
\end{eqnarray}
Then $\{p_{km}(s)\}_{k=0}^\mathcal{N}$ are WOPs with the weight function $(s-a)^m$.
Furthermore, (\ref{l062}) can be written as the following matrix form:
\begin{eqnarray}\label{l083}
F_\mathcal{N}(s)=G_{\mathcal{N}m}P_{\mathcal{N}m}(s)
\end{eqnarray}
with
\begin{eqnarray}\label{l084}
P_{\mathcal{N}m}(s)&=&\mathrm{col}(p_{0m}(s),p_{1m}(s),\dots,p_{\mathcal{N}m}(s)),\nonumber
\end{eqnarray}
\begin{eqnarray}\label{15sep082}
F_{\mathcal{N}}(s)&=&\mathrm{col}(f_0(s),f_1(s),\dots,f_{\mathcal{N}}(s)),\nonumber
\end{eqnarray}
where $G_{\mathcal{N}m}$ be the $(\mathcal{N}+1)\times (\mathcal{N}+1)$ unit lower triangular matrix with the $(i,j)$-th entry equal to $\frac{g_{i-1,j-1,m}}{\chi_{j-1,m}}$ for any $i>j$.

\subsection{WOPs-based integral inequalities}

To prove WOPs-based integral inequalities, the following property on Kronecker product of matrices is required.

\begin{lemma}\label{15sep11-lemma-1}\cite{Horn(1985)09051182}
If $A$, $B$, $C$ and $D$ are matrices of appropriate sizes, then $(A \otimes B)(C \otimes D) = (AC) \otimes (BD)$.
\end{lemma}

Based on the previous preparation, now we can investigate the following WOPs-based integral inequalities which give lower bounds of $\mathcal{I}_m(w_t)$.

\begin{theorem}\label{15sep08-theorem-1}
For given integers $\mathcal{N}\ge 0$ and $m\ge 0$, a symmetric positive definite matrix $R\in \mathbb{R}^{n\times n}$, and a continuous function $\omega:[a,b]\rightarrow \mathbb{R}^n$, the following inequality holds:
\begin{eqnarray}\label{15aug0427}
\mathcal{I}_m(w_t)\ge
F_{\mathcal{N}m}^\mathrm{T}(w_t)(\Xi_{\mathcal{N}m} \otimes R)F_{\mathcal{N}m}(w_t)
\end{eqnarray}
with
\begin{eqnarray}\label{15aug0426}
F_{\mathcal{N}m}(w_t)=\int_{a}^b(s-a)^m(F_\mathcal{N}(s)\otimes w_t(s))\mathrm{d}s,
\end{eqnarray}
\begin{eqnarray}\label{15aug0428}
\Xi_{\mathcal{N}m}=G_{\mathcal{N}m}^{-\mathrm{T}}\Lambda_{\mathcal{N}m}^{-1}G_{\mathcal{N}m}^{-1},
\end{eqnarray}
\begin{eqnarray}\label{15aug0414}
\Lambda_{\mathcal{N}m}=\mathrm{diag}(\chi_{0m},\chi_{1m},\chi_{2m},\dots,\chi_{\mathcal{N}m}),
\end{eqnarray}
and $F_\mathcal{N}(s)$, $\chi_{km}$ and $G_{\mathcal{N}m}$ are defined as previously.
\end{theorem}

\begin{proof}
Set
\begin{eqnarray}\label{15aug0420}
z(s)=w_t(s)-\sum_{k=0}^\mathcal{N}\chi_{km}^{-1} p_{km}(s)\pi_{km}(w_t)\nonumber
\end{eqnarray}
with
\begin{eqnarray}\label{15aug0413}
\pi_{km}(w_t)=\int_{a}^b(s-a)^mp_{km}(s)w_t(s)\mathrm{d}s.
\end{eqnarray}
Then it follows from (\ref{15aug161}), (\ref{l072}) and the orthogonality of $\{p_{km}(s)\}_{k=0}^\mathcal{N}$ under the weight function $(s-a)^m$ that
\begin{eqnarray}\label{15aug0422}
\mathcal{I}_m(z)=\mathcal{I}_m(w_t)-\sum_{k=0}^\mathcal{N}\chi_{km}^{-1}\pi_{km}^\mathrm{T}(w_t)R\pi_{km}(w_t).\nonumber
\end{eqnarray}
This, together with $\mathcal{I}_m(z)\ge 0$, implies that
\begin{eqnarray}\label{15aug0423}
\mathcal{I}_m(w_t)&\ge&\sum_{k=0}^\mathcal{N}\chi_{km}^{-1}\pi_{km}^\mathrm{T}(w_t)R\pi_{km}(w_t)\nonumber\\
&=&\Pi_{\mathcal{N}m}^\mathrm{T}(w_t)(\Lambda_{\mathcal{N}m}^{-1}\otimes R)\Pi_{\mathcal{N}m}(w_t),
\end{eqnarray}
where
\begin{eqnarray}\label{15aug0412}
\Pi_{\mathcal{N}m}(w_t)=\mathrm{col}(\pi_{0m}(w_t),\pi_{1m}(w_t),\dots,\pi_{\mathcal{N}m}(w_t)). \nonumber
\end{eqnarray}
Since $G_{\mathcal{N}m}$ is a unit lower triangular matrix, it follows from (\ref{l083}), (\ref{15aug0426}) and (\ref{15aug0413}) that
\begin{eqnarray}\label{15aug0424}
\Pi_{\mathcal{N}m}(w_t)
&=&
\int_{a}^b(s-a)^m(P_{\mathcal{N}m}(s)\otimes w_t(s))\mathrm{d}s\nonumber\\
&=&\int_{a}^b(s-a)^m(G_{\mathcal{N}m}^{-1}F_\mathcal{N}(s)\otimes w_t(s))\mathrm{d}s\nonumber\\
&=&(G_{\mathcal{N}m}^{-1}\otimes I_n)F_{\mathcal{N}m}(w_t).\nonumber
\end{eqnarray}
This, together with (\ref{15aug0423}) and Lemma \ref{15sep11-lemma-1}, completes the proof.
\end{proof}

Since the inequality (\ref{15aug0427}) is obtained by using the WOPs (\ref{l062}), we will refer to (\ref{15aug0427}) as \emph{WOPs-based integral inequalities}.

\section{Discussions}\label{15sep11-section-1}

In this section we will discuss the relation between the WOPs-based integral inequalities in Theorem \ref{15sep08-theorem-1} and the JIs and WTIIs in
\cite{Gu(2000),IJRNC-2009-1364,AMC-2013-714,A-2013-2860,ECC-2014-448,CDC-2013-946,SCL-2015-1,ECC-2013,
A-2015-204,JFI-2015stability,IEEE-T-AC-2015-free,A-2015-189,CNSNS-2013-1246,DDNS-2013-793686,JFI-2015-1378,JFI-2015enhanced,IFAC-2012}.

When $(\mathcal{N},m)=(2,0)$, by employing the symbolic operations of MATLAB, one can easily check that
\begin{eqnarray}\label{15sep083}
F_2(s)=\begin{bmatrix}
         1\\
     s-a\\
 (s-a)^2
 \end{bmatrix},\
G_{20}^{-1}=\begin{bmatrix}
     1&  0& 0\\
  \frac{a-b}{2}&  1& 0\\
 \frac{(b-a)^2}{6}& a-b& 1
\end{bmatrix},\nonumber
\end{eqnarray}
\begin{eqnarray}\label{15sep084}
\Lambda_{20}=\frac{1}{b-a}\mathrm{diag}(1,\frac{12}{(b-a)^2},\frac{180}{(b-a)^4}),\nonumber
\end{eqnarray}
and hence
\begin{eqnarray}\label{15sep085}
F_{20}(w_t)=\begin{bmatrix}
\int_{a}^bw_t(s)\mathrm{d}s\\
\int_{a}^b\int_{\alpha}^bw_t(s)\mathrm{d}s\mathrm{d}\alpha\\
2\int_{a}^b\int_{\beta}^b\int_{\alpha}^bw_t(s)\mathrm{d}s\mathrm{d}\alpha\mathrm{d}\beta\nonumber
\end{bmatrix},
\end{eqnarray}
\begin{eqnarray}\label{15sep086}
\Xi_{20}(R)=\frac{1}{b-a}\left(\delta_1R\delta_1^\mathrm{T}+3\delta_2R\delta_2^\mathrm{T}+5\delta_3R\delta_3^\mathrm{T}\right),\nonumber
\end{eqnarray}
where
\begin{eqnarray}\label{15sep087}
\delta_1=\begin{bmatrix}
1\\
0\\
0
\end{bmatrix},\ \delta_2=\begin{bmatrix}
1\\
\frac{2}{a-b}\\
0
\end{bmatrix},\ \delta_3=\begin{bmatrix}
1\\
\frac{6}{a-b}\\
\frac{6}{(b-a)^2}
\end{bmatrix}.\nonumber
\end{eqnarray}
This, together with Theorem \ref{15sep08-theorem-1}, yields the following result.

\begin{corollary}\label{15sep08-corollary-1}
When $(\mathcal{N},m)=(2,0)$,  the inequality (\ref{15aug0427}) turns into \cite[(13)]{JFI-2015-1378}, that is,
\begin{eqnarray}\label{15sep088}
\mathcal{I}_0(w_t)\ge
\frac{1}{b-a}\left(\Omega_0^\mathrm{T}R\Omega_0+3\Omega_1^\mathrm{T}R\Omega_1+5\Omega_2^\mathrm{T}R\Omega_2\right),
\end{eqnarray}
where
\begin{eqnarray}\label{15sep089}
\Omega_0&=&\int_{a}^bw_t(s)\mathrm{d}s,\nonumber\\
\Omega_1&=&\int_{a}^bw_t(s)\mathrm{d}s-\frac{2}{b-a}\int_{a}^b\int_{\alpha}^bw_t(s)\mathrm{d}s\mathrm{d}\alpha,\nonumber\\
\Omega_2&=&\int_{a}^bw_t(s)\mathrm{d}s-\frac{6}{b-a}\int_{a}^b\int_{\alpha}^bw_t(s)\mathrm{d}s\mathrm{d}\alpha\nonumber\\
&&+\frac{12}{(b-a)^2}\int_{a}^b\int_{\beta}^b\int_{\alpha}^bw_t(s)\mathrm{d}s\mathrm{d}\alpha\mathrm{d}\beta.\nonumber
\end{eqnarray}
\end{corollary}

Similar to Corollary \ref{15sep08-corollary-1}, the following several corollaries can be easily derived from Theorem \ref{15sep08-theorem-1}.

\begin{corollary}\label{15sep08-corollary-2}
When $(\mathcal{N},m)=(1,0)$, the inequality (\ref{15aug0427}) turns into the so-called Wirtinger-based integral inequality \cite[(8)]{JFI-2015-1378}, that is,
\begin{eqnarray}\label{15sep0810}
\mathcal{I}_0(w_t)\ge
\frac{1}{b-a}\left(\Omega_0^\mathrm{T}R\Omega_0+3\Omega_1^\mathrm{T}R\Omega_1\right),
\end{eqnarray}
where $\Omega_0$ and $\Omega_1$ are defined as in Corollary \ref{15sep08-corollary-1}.
\end{corollary}

\begin{corollary}\label{15sep08-corollary-3}
When $\mathcal{N}=0$, the inequality (\ref{15aug0427}) turns into the celebrated Jensen's inequalities
(see \cite{Gu(2000)} and \cite{IJRNC-2009-1364} for the cases $m=0$ and $m=1$, respectively; and \cite[Lemma 1]{AMC-2013-714} for the special case $(a,b)=(-d,0)$), that is,
\begin{eqnarray}\label{15aug262}
&&\int_{a}^b\!\int_{\theta_1}^b\!\!\cdots\!\!\int_{\theta_m}^b\!\!w_t^\mathrm{T}(s)Rw_t(s)\mathrm{d}s\mathrm{d}\theta_m\cdots\mathrm{d}\theta_{1}\nonumber\\
&\ge&
\frac{(m+1)!}{(b-a)^{m+1}}\tilde{\Omega}_{m}^\mathrm{T}R\tilde{\Omega}_{m},
\end{eqnarray}
where
\begin{eqnarray}\label{15sep091}
\tilde{\Omega}_{m}=\int_{a}^b\!\int_{\theta_1}^b\!\!\cdots\!\!\int_{\theta_m}^b\!\!w_t(s)\mathrm{d}s\mathrm{d}\theta_m\cdots\mathrm{d}\theta_{1}.
\end{eqnarray}
\end{corollary}

\begin{corollary}\label{15sep12-corollary-1}
When $\mathcal{N}=1$, the inequality (\ref{15aug0427}) turns into
\begin{eqnarray}\label{15sep121}
&&\int_{a}^b\!\int_{\theta_1}^b\!\!\cdots\!\!\int_{\theta_m}^b\!\!w_t^\mathrm{T}(s)Rw_t(s)\mathrm{d}s\mathrm{d}\theta_m\cdots\mathrm{d}\theta_{1}\nonumber\\
&\ge&
\frac{(m+1)!}{(b-a)^{m+1}}\left(\tilde{\Omega}_{m}^\mathrm{T}R\tilde{\Omega}_{m}
+(m+3)(m+1)\Sigma_m^\mathrm{T}R\Sigma_m\right),\nonumber
\end{eqnarray}
where
\begin{eqnarray}\label{15sep131}
\Sigma_m=\tilde{\Omega}_{m}-\frac{m+2}{b-a}\tilde{\Omega}_{m+1},\nonumber
\end{eqnarray}
and $\tilde{\Omega}_m$ and $\tilde{\Omega}_{m+1}$ are defined as in Corollary \ref{15sep08-corollary-3}.
\end{corollary}

\begin{corollary}\label{15sep08-corollary-4}
When $(\mathcal{N},m)=(1,1)$, the inequality (\ref{15aug0427}) turns into \cite[(16)]{JFI-2015-1378}, that is,
\begin{eqnarray}\label{15sep0812}
\mathcal{I}_1(w_t)\ge
\frac{2}{(b-a)^2}\left(\Omega_3^\mathrm{T}R\Omega_3+8\Omega_4^\mathrm{T}R\Omega_4\right),
\end{eqnarray}
where
\begin{eqnarray}\label{15sep0813}
\Omega_3&=&\int_{a}^b\hspace*{-1mm}\int_{\alpha}^b\hspace*{-1mm}w_t(s)\mathrm{d}s\mathrm{d}\alpha,\nonumber\\
\Omega_4&=&\int_{a}^b\hspace*{-1mm}\int_{\alpha}^b\hspace*{-1mm}w_t(s)\mathrm{d}s\mathrm{d}\alpha-\frac{3}{b-a}\int_{a}^b\hspace*{-1mm}\int_{\beta}^b\hspace*{-1mm}\int_{\alpha}^b\hspace*{-1mm}w_t(s)\mathrm{d}s\mathrm{d}\alpha\mathrm{d}\beta.\nonumber
\end{eqnarray}
\end{corollary}

\begin{corollary}\label{15sep10-corollary-3}
When $m=0$ and $(a,b)=(-h,0)$, the inequality (\ref{15aug0427}) turns into the so-called Bessel--Legendre inequality \cite[Lemma 3]{SCL-2015-1} (i.e., \cite[Lemma 3]{ECC-2014-448}), that is,
\begin{eqnarray}\label{15sep102}
\mathcal{I}_0(w_t)\ge \frac{1}{h}\sum_{k=0}^\mathcal{N}(2k+1)\hat{\Omega}_k^\mathrm{T}R\hat{\Omega}_k,
\end{eqnarray}
where $\hat{\Omega}_k=\int_{-h}^0L_k(s)w_t(s)\mathrm{d}s$, and $\{L_k(s)\}_{k=0}^\mathcal{N}$ is the Legendre orthogonal polynomials defined in \cite[Definition 1]{SCL-2015-1}.
\end{corollary}

If we replace $w_t$ by $\dot{w}_t$ in Corollaries \ref{15sep08-corollary-1}--\ref{15sep08-corollary-4}, then the following several results can be obtained.

\begin{corollary}\label{15sep09-corollary-1}
When $(\mathcal{N},m)=(2,0)$ and $w_t$ is replaced by $\dot{w}_t$,  the inequality (\ref{15aug0427}) turns into \cite[(24)]{JFI-2015-1378}, that is,
\begin{eqnarray}\label{15sep098}
\mathcal{I}_0(\dot{w}_t)\ge
\frac{1}{b-a}\left(\Theta_0^\mathrm{T}R\Theta_0+3\Theta_1^\mathrm{T}R\Theta_1+5\Theta_2^\mathrm{T}R\Theta_2\right),
\end{eqnarray}
where
\begin{eqnarray}\label{15sep099}
\Theta_0&=&w_t(b)-w_t(a),\nonumber\\
\Theta_1&=&w_t(b)+w_t(a)-\frac{2}{b-a}\int_{a}^bw_t(s)\mathrm{d}s,\nonumber\\
\Theta_2&=&w_t(b)-w_t(a)+\frac{6}{b-a}\int_{a}^bw_t(s)\mathrm{d}s\nonumber\\
&&-\frac{12}{(b-a)^2}\int_{a}^b\int_{\alpha}^bw_t(s)\mathrm{d}s\mathrm{d}\alpha.\nonumber
\end{eqnarray}
\end{corollary}

\begin{corollary}\label{15sep09-corollary-2}
When $(\mathcal{N},m)=(1,0)$ and $w_t$ is replaced by $\dot{w}_t$,  the inequality (\ref{15aug0427}) turns into \cite[Corollary 5]{A-2013-2860}(i.e., \cite[Lemma 2.1]{ECC-2013} and \cite[Lemma 2.1]{CDC-2013-946} or \cite[(23)]{JFI-2015-1378}), that is,
\begin{eqnarray}\label{15sep0910}
\mathcal{I}_0(\dot{w}_t)\ge
\frac{1}{b-a}\left(\Theta_0^\mathrm{T}R\Theta_0+3\Theta_1^\mathrm{T}R\Theta_1\right),
\end{eqnarray}
where $\Theta_0$ and $\Theta_1$ are defined as in Corollary \ref{15sep09-corollary-1}.
\end{corollary}

\begin{corollary}\label{15sep09-corollary-3}
When $\mathcal{N}=0$ and $w_t$ is replaced by $\dot{w}_t$, the inequality (\ref{15aug0427}) turns into the celebrated Jensen's inequalities (see \cite{Gu(2000)} and \cite{IJRNC-2009-1364} for the cases $m=0$ and $m=1$, respectively), that is,
\begin{eqnarray}\label{15sep093}
&&\int_{a}^b\!\int_{\theta_1}^b\!\!\cdots\!\!\int_{\theta_m}^b\!\!\dot{w}_t^\mathrm{T}(s)R\dot{w}_t(s)\mathrm{d}s\mathrm{d}\theta_m\cdots\mathrm{d}\theta_{1}\nonumber\\
&\ge&
\frac{(m+1)!}{(b-a)^{m+1}}\tilde{\Theta}_m^\mathrm{T}R\tilde{\Theta}_m,
\end{eqnarray}
where $\tilde{\Theta}_0=w_t(b)-w_t(a)$ and
\begin{eqnarray}\label{15sep092}
\tilde{\Theta}_m=\frac{(b-a)^m}{m!}w_t(b)-\tilde{\Omega}_{m-1}, m\ge 1.\nonumber
\end{eqnarray}
\end{corollary}

\begin{corollary}\label{15sep12-corollary-2}
When $\mathcal{N}=1$ and $w_t$ is replaced by $\dot{w}_t$, the inequality (\ref{15aug0427}) turns into
\begin{eqnarray}\label{15sep122}
&&\int_{a}^b\!\int_{\theta_1}^b\!\!\cdots\!\!\int_{\theta_m}^b\!\!\dot{w}_t^\mathrm{T}(s)R\dot{w}_t(s)\mathrm{d}s\mathrm{d}\theta_m\cdots\mathrm{d}\theta_{1}\nonumber\\
&\ge&
\frac{(m+1)!}{(b-a)^{m+1}}\left(\tilde{\Theta}_m^\mathrm{T}R\tilde{\Theta}_m
+(m+1)(m+3)\Psi_m^\mathrm{T}R\Psi_m\right),\nonumber
\end{eqnarray}
where
\begin{eqnarray}\label{15sep132}
\Psi_m=-\frac{(b-a)^m}{(m+1)!}w_t(b)-\tilde{\Theta}_{m-1}+\frac{m+2}{b-a}\tilde{\Theta}_{m},\nonumber
\end{eqnarray}
and $\tilde{\Theta}_m$ and $\tilde{\Theta}_{m-1}$ are defined as in Corollary \ref{15sep09-corollary-3}.
\end{corollary}

\begin{corollary}\label{15sep09-corollary-4}
When $(\mathcal{N},m)=(1,1)$ and $w_t$ is replaced by $\dot{w}_t$,  the inequality (\ref{15aug0427}) turns into \cite[(25)]{JFI-2015-1378}, that is,
\begin{eqnarray}\label{15sep0912}
\mathcal{I}_1(\dot{w})\ge
2\Theta_3^\mathrm{T}R\Theta_3+4\Theta_4^\mathrm{T}R\Theta_4,
\end{eqnarray}
where
\begin{eqnarray}\label{15sep0913}
\Theta_3&=&w_t(b)\hspace*{-1mm}-\hspace*{-1mm}\frac{1}{b-a}\int_{a}^bw_t(s)\mathrm{d}s,\nonumber\\
\Theta_4&=&w_t(b)\hspace*{-1mm}+\hspace*{-1mm}\frac{2}{b-a}\int_{a}^b\hspace*{-2mm}w_t(s)\mathrm{d}s\hspace*{-1mm}-\hspace*{-1mm}\frac{6}{(b-a)^2}\int_{a}^b\hspace*{-2mm}\int_{\alpha}^b\hspace*{-2mm}w_t(s)\mathrm{d}s\mathrm{d}\alpha.\nonumber
\end{eqnarray}
\end{corollary}

\begin{remark}\label{15sep10-remark-4}
Based on (\ref{15sep0810}), Park \textit{et al.} \cite[Corollary 1]{A-2015-204} derived
\begin{eqnarray}\label{15sep0914}
\mathcal{I}_1(w_t)\ge
\frac{2}{(b-a)^2}\left(\Omega_3^\mathrm{T}R\Omega_3+2\Omega_4^\mathrm{T}R\Omega_4\right).
\end{eqnarray}
Clearly, the inequality (\ref{15sep0812}) is more accurate than (\ref{15sep0914}).
\end{remark}

\begin{remark}\label{15sep10-remark-2}
It has been proven by Gyurkovics in \cite[Corollary 9]{A-2015-44} that Corollary \ref{15sep09-corollary-2} is equivalent to \cite[Lemma 4]{IEEE-T-AC-2015-free} in term of establishing stability criteria for delayed continuous-time systems. By a similar approach, we can show that Corollary \ref{15sep09-corollary-1} is equivalent to \cite[Lemma 1]{A-2015-189}.
However, unlike \cite[Lemma 4]{IEEE-T-AC-2015-free} and \cite[Lemma 1]{A-2015-189}, no free-weighting matrix is involved in Corollaries \ref{15sep09-corollary-1} and \ref{15sep09-corollary-2}.
\end{remark}

\begin{remark}\label{15sep10-remark-1}
It has been proven by Gyurkovics in \cite[Theorem 6]{A-2015-44} that the inequality in \cite[Lemma 2.4]{CNSNS-2013-1246} (i.e., \cite[(12)]{DDNS-2013-793686}) is more conservative than (\ref{15sep0910}). By a similar approach, one can prove that the inequality \cite[(13)]{DDNS-2013-793686}) is more conservative than (\ref{15sep0912}).
\end{remark}

\begin{remark}
The inequalities in Corollaries \ref{15sep12-corollary-1} and \ref{15sep12-corollary-2} are more accurate than ones in \cite[Lemmas 5 and 6]{JFI-2015stability}, respectively,
since the coefficients of the second item on the right-hand side of the inequalities in \cite[Lemmas 5 and 6]{JFI-2015stability} is $\frac{m!(m+3)}{(b-a)^{m+1}}$ which is smaller than
$\frac{m!(m+1)^2(m+3)}{(b-a)^{m+1}}$ in Corollaries \ref{15sep12-corollary-1} and \ref{15sep12-corollary-2}.
\end{remark}

\begin{remark}\label{15sep11-remark-1}
Note that Corollary \ref{15sep09-corollary-2} refines the inequality proposed
in \cite[Lemma 5]{IFAC-2012}, in which the second term of the righthand side is $\frac{\pi^2}{4}\Theta_1^\mathrm{T}R\Theta_1$ which is less than or equal to $3\Theta_1^\mathrm{T}R\Theta_1$. So,
Corollary \ref{15sep09-corollary-2} is less conservative than \cite[Lemma 5]{IFAC-2012}.
\end{remark}

\begin{remark}\label{15sep10-remark-3}
If $(s-a)^k$ is replaced by $(b-s)^k$ for all positive integer $k$ throughout this paper, then we can obtain new WOPs-based integral inequalities like (\ref{15aug0427}), which is a generalization of
\cite[Corollary 4]{A-2013-2860}, \cite[(3.1) and (3.8)]{JFI-2015enhanced} and \cite[(18) and (26)]{JFI-2015-1378}.
\end{remark}

Corollaries \ref{15sep08-corollary-1}--\ref{15sep09-corollary-4} imply that Theorem \ref{15sep08-theorem-1} contains the corresponding results of \cite{Gu(2000),IJRNC-2009-1364,AMC-2013-714,JFI-2015-1378,A-2013-2860,ECC-2014-448,CDC-2013-946,SCL-2015-1,ECC-2013} as special cases,
while Remarks \ref{15sep10-remark-4}--\ref{15sep10-remark-3} present that Theorem \ref{15sep08-theorem-1} improves the corresponding results of \cite{A-2015-204,JFI-2015stability,IEEE-T-AC-2015-free,A-2015-189,CNSNS-2013-1246,DDNS-2013-793686,JFI-2015-1378,JFI-2015enhanced,IFAC-2012}.
Therefore, Theorem \ref{15sep08-theorem-1} is a generalization of these literature.

\section{Conclusion}\label{15jun06-section-1}

In this paper, we have provided WOPs-based integral inequalities
which encompass and/or improve the corresponding inequalities in \cite{Gu(2000),IJRNC-2009-1364,AMC-2013-714,A-2013-2860,ECC-2014-448,CDC-2013-946,SCL-2015-1,ECC-2013,
A-2015-204,JFI-2015stability,IEEE-T-AC-2015-free,A-2015-189,CNSNS-2013-1246,DDNS-2013-793686,JFI-2015-1378,JFI-2015enhanced,IFAC-2012}.
From these literature, it is clear that the WOPs-based integral inequalities obtained in this paper have potential applications in establishing less conservative stability criteria for delayed continuous-time systems. This will be proceeded in our future work.


\end{document}